\documentclass[11pt, reqno]{amsart}
\usepackage{amssymb,amsmath,amscd, amsfonts,latexsym}
\usepackage[top = 3cm, bottom = 3cm, left = 2.5cm, right = 2.5cm]{geometry}
\usepackage{hyperref}

\newtheorem{theorem}{Theorem}

\theoremstyle{remark}

\title[\title{Shorey--Tijdeman equation with Lucas sequences}]{On the  Shorey--Tijdeman Diophantine equation involving terms of Lucas sequences$ ^{*} $}

\author[M. Ddamulira, F. Luca, \& R. Tichy]{Mahadi Ddamulira,  Florian Luca, AND Robert Tichy}



\address{Mahadi Ddamulira \newline
         \indent Department of Mathematics, College of Natural Sciences, Makerere University, \newline
         \indent Kampala, Uganda}
\address{Max Planck Institute for Software Systems, Saarland Informatics Campus, \newline
\indent Saarbr\"ucken, Germany}

\email{mahadi.ddamulira\char'100mak.ac.ug; mddamulira\char'100mpi-sws.org}

\address{Florian Luca \newline
         \indent School of Mathematics, University of the Witwatersrand, \newline
        \indent Johannesburg, South Africa}
         
\address{Research Group in Algebraic Structures and Applications, King Abdulaziz University,\newline 
         \indent Jeddah, Saudi Arabia}

\address{Max Planck Institute for Software Systems, Saarland Informatics Campus, \newline
\indent  Saarbr\"ucken, Germany}           

\address{Centro de Ciencias Matem\'aticas UNAM\newline
\indent Morelia, Mexico}

\email{Florian.Luca\char'100wits.ac.za}

\address{Robert Tichy \newline
         \indent Institute of Analysis and Number Theory, Graz University of Technology, \newline
         \indent Graz, Austria}

\email{tichy\char'100tugraz.at}

\thanks{$ ^{*} $This note is dedicated to Robert Tijdeman on the occasion of his 80th birthday \\
\indent $ ^{**} $Accepted to appear in Indagationes Mathematicae}
\date{}

\begin{document}

% ***********************************************************************************

\begin{abstract}
Let  $r\ge 1$ be an integer and ${\bf U}:=\{U_n\}_{n\ge 0}$ be the Lucas sequence given by $U_0=0,~U_1=1$, and $U_{n+2}=rU_{n+1}+U_n$ for $n\ge 0$. In this paper, we explain how to find all the solutions of the Diophantine equation, $AU_{n}+BU_{m}=CU_{n_1}+DU_{m_1}$, in integers $r\ge 1$, $0\le m<n,~0\le m_1<n_1$, $AU_n\ne CU_{n_1}$, where $A,B,C,D$ are given integers with $A\ne 0,~B\ne 0$, $m,n,m_1,n_1$ are nonnegative integer unknowns and $r$ is also unknown.
\end{abstract}

\maketitle
% ***********************************************************************************

\section{Introduction}
Let $r\ge 1$ be an integer and ${\bf U}:=(U_{n})_{n\ge 0} $ be the Lucas sequence given by $ U_0=0, ~ U_1=1, $ and 
\begin{align}
U_{n+2}=rU_{n+1}+U_n
\end{align} for all $ n\ge 0$. When $r=1$, ${\bf U}$ coincides with the Fibonacci sequence while when $r=2$, ${\bf U}$ coincides with the Pell sequence.

Let $$ (\alpha, \beta):= \left(\dfrac{r+\sqrt{r^2+4}}{2},  \dfrac{r-\sqrt{r^2+4}}{2}\right), $$
be the roots of the characteristic equation $ X^2-rX-1=0 $ of the Lucas sequence ${\bf U}=(U_n)_{n\ge 0}$. It is easy to see  that $\beta=-\alpha^{-1}$.  The Binet formula for the general term of ${\bf U}$ is given by
\begin{align}\label{binet}
U_n := \dfrac{\alpha^{n}-\beta^{n}}{\alpha-\beta} \quad \text{for all} \quad n\ge 0.
\end{align}
The divisibility property 
\begin{align}\label{gcd}
\gcd(U_n,U_m)  =  U_{\gcd(n,m)}\quad {\text{\rm for~positive~integers}}\quad n,m
\end{align}
is well-known. It is heavily used in solving Diophantine equations involving members of Lucas sequences and it is an important ingredient in the proof of the Primitive Divisor Theorem for Lucas sequences (see \cite{BHV} for such properties. In particular, the above property \eqref{gcd}  appears as Proposition 2.1 (iii) in \cite{BHV}). Furthermore, one can prove by induction that the inequality
\begin{align}\label{probb}
\alpha^{n-2}\le U_n \le \alpha^{n-1}
\end{align}
holds for all positive integers $n$.

Shorey and Tijdeman \cite{Shorey} gave lower bounds for the absolute value and the greatest prime factor of the expression $ Ax^m+By^m $ where $ A,B,x,y,m\ge 0 $ are integers. As an application, they proved, under suitable conditions, that the equation $ Ax^m+By^m=Cx^n+Dy^n $ implies that $ \max\{n,m\} $ is bounded by a computable constant depending only on $ A,B,C, D $. More precisely, they proved the following result.
\begin{theorem}\label{Shorey-Tijdeman}
Let $ A\neq 0 $, $ B\neq 0 $, $ C $, and $ D $ be integers. Suppose that $ x,y,m,n $ with $ |x|\neq |y| $ and $ 0\le n<m $ are integers. There exists a computable constant $ E $ depending only on $ A,B,C,D $ such that the Diophantine equation
\begin{align}\label{ST}
Ax^m+By^m = Cx^n+Dy^n
\end{align}
with
\begin{align}
Ax^m \neq Cx^n
\end{align}
implies that $ m\le E $.
\end{theorem}
\noindent
In this paper, we study a variation of the above result with the terms of the Lucas sequence ${\bf U}:=(U_{n})_{n\ge 0} $. That is, we study the Diophantine equation
\begin{equation}
\label{eq:main}
AU_n+BU_m=CU_{n_1}+DU_{m_1}\quad \text{with} \quad  n>m\ge 0\quad {\text{and}}\quad n_1>m_1\ge 0,\quad AU_n\ne CU_{n_1}.
\end{equation}
Our first result is the following.
\begin{theorem}
Assume that $A,B,C,D$ are given integers, $AB\ne 0$ and Eq. \eqref{eq:main} holds. Then $r<14X$, where $X:=\max\{|A|,|B|,|C|,|D|\}$.
\end{theorem}
\begin{proof} 
Assume first that $C=D=0$. Then we take $m_1=0,~n_1=1$. Then $AU_n=-BU_m$. If $m=0$, 
then $n=0$ which is not allowed. Thus, $m\ne 0$,  so $U_n/U_d$ divides $B$, where $d:=\gcd(n,m)$. Write $n=:kd$, where $k\ge 2$. If $d=1$, then $U_n/U_d=U_k/U_1=U_k\ge U_2=r$, so $r\le X$. If $d\ge 2$, then 
\begin{align*}
\frac{U_n}{U_d}=\frac{\alpha^{kd}-\beta^{kd}}{\alpha^d-\beta^d}.
\end{align*}
We show that this last expression is $>\alpha$. This is equivalent to 
\begin{align*}
\alpha^{kd}>\alpha^{d+1}-\alpha\beta^d+\beta^{kd}.
\end{align*}
Since $d\ge 2$, $|\alpha\beta^{d}|=|\beta|^{d-1}<1$. Thus, it suffices that 
\begin{align*}
\alpha^{2d}-\alpha^{d+1}>2.
\end{align*}
The left--hand side is $\alpha^{d+1}(\alpha^{d-1}-1)\ge \alpha^{d+1}(\alpha-1)$. The smallest possible $\alpha$ is $\phi:=(1+{\sqrt{5}})/2$ (for $r=1$) and $\phi^{d+1}(\phi-1)\ge \phi^3(\phi-1)=\phi^2>2$. 
Thus, indeed $\alpha<U_{kd}/U_d\le X$, which gives $r=\alpha+\beta<\alpha<X$. Further, $U_n\ge \alpha^{n-2}$ and  $U_d\le \alpha^{d-1}$ (by \eqref{probb}), so 
\begin{align*}
\frac{U_n}{U_d}\ge \alpha^{n-d-3}\ge \alpha^{n-n/2-3}\ge \alpha^{n/2-3}.
\end{align*}
 In the above we used that 
$d<n$ is a proper divisor of $n$, so $d\le n/2$. Since $U_n/U_d$ divides $B$, we get that  $\alpha^{n/2-3}\le |B|\le X$. Since $\alpha\ge \phi$, we get
\begin{equation}
\label{eq:CDarezero}
0<m<n\le 6+ 2\frac{\log X}{\log \phi}.
\end{equation}
This is when $C=D=0$. 

So, we may assume that not both $C,~D$ are $0$. If one of $C,~D$ is nonzero and the other is zero, we assume that $C\ne 0$ and $n_1\ne 0$. Thus, if either $D=0$ or $m_1=0$, then the right--hand side 
is $CU_{n_1}$, otherwise it is $CU_{n_1}+DU_{m_1}$ with $D\ne 0$ and $n_1>m_1>0$. If $n=n_1$, then 
$$
(A-C)U_n+BU_m=DU_{m_1}.
$$
The case $A-C=0$ is not allowed since then $AU_n=AU_{n_1}$. Thus, $A-C\ne 0$ and also $D\ne 0$. We also assume that $m\ne 0$ since if $m=0$, we are in the preceeding case. So, if $n=n_1$, then we replace $(A,B,C,D)$ by $(A-C,B,D,0)$. The only effect is that $X$ is replaced by $2X$. Thus, 
we may assume that $n\ne n_1$, and switching $A$ with $C$, if needed, we may assume that $n=\max\{n,n_1\}$, therefore $n>n_1$. We relabel our indices $(n,m,n_1,m_1)$ as $(n_1,n_2,n_3,n_4)$ where 
$n_1>n_2\ge n_3\ge n_4$, and the coefficients $A,B,C,D$ as $A_1,A_2,A_3,A_4$ and change signs to at most a couple of them so that our equation is now
\begin{equation}
\label{eq:1}
A_1U_{n_1}+A_2U_{n_2}+A_3 U_{n_3}+A_{4} U_{n_4}=0.
\end{equation}
Furthermore, $A_1,~A_2,~A_3$ are all nonzero but $A_4$ (or $n_4$) might be $0$. This leads to
\begin{align*}
|A_1|\alpha^{n_1}=|-A_2(\alpha^{n_2}-\beta^{n_2})-A_3(\alpha^{n_3}-\beta^{n_3})-A_4(\alpha^{n_4}-\beta^{n_4})+A_1\beta^{n_1}|<7X\alpha^{n_2},
\end{align*}
so 
\begin{equation}
\label{eq:n1minusn2}
\alpha^{n_1-n_2}<7X.
\end{equation}
 Thus, since $n_1>n_2$, we get that $r<\alpha\le \alpha^{n_1-n_2}<7X$. Recalling that we might have to replace $X$ by $2X$, we get the desired conclusion.
\end{proof}

\section{Finding all solutions} 
So far, we know that $r$ is bounded. It is possible for small $r$ that the equation has infinitely many solutions. By the preceding analysis, we saw that this is not the case if $C=D=0$, since then 
$0<n<6+2\log X/\log \phi$.  So, we assume that not both $C$ and $D$ are zero. Using the substitution $(A,B,C,D)\mapsto (A-C,B,D,0)$, and relabelling some of the variables, we may assume that 
$n_1>n_1\ge n_3\ge n_4$ and that equation \eqref{eq:1} holds. Then estimate \eqref{eq:n1minusn2} holds, so 
\begin{align*}
n_1-n_2<\frac{\log(7X)}{\log \phi}.
\end{align*}
We return to \eqref{eq:1} and rewrite it as
\begin{equation}
\label{eq:3}
\left|\alpha^{n_2}(A_1\alpha^{n_1-n_2}+A_2)-\left(\frac{A_1}{(-\alpha)^{n_1}}+\frac{A_2}{(-\alpha)^{n_2}}\right)\right|=|-A_3(\alpha^{n_3}-\beta^{n_3})-A_4(\alpha^{n_4}-\beta^{n_4})|.
\end{equation}
The right--hand side is $\le 4X\alpha^{n_3}$. In the left--hand side we have $n_1-n_2>0$, so $A_1\alpha^{n_1-n_2}+A_2\ne 0$. Thus, 
$$
|A_1\alpha^{n_1-n_2}+A_2||A_1\beta^{n_1-n_2}+A_2|\ge 1.
$$
The second factor in the left above is $\le 2X$. Thus, $|A_1\alpha^{n_1-n_2}+A_2|\ge 1/2X$. Further, 
$$
\left| \frac{A_1}{(-\alpha)^{n_1}}+\frac{A_2}{(-\alpha)^{n_2}}\right|\le \frac{2X}{\alpha^{n_2}}.
$$
Hence, 
$$
\left|\alpha^{n_2}(A_1\alpha^{n_1-n_2}+A_2)-\left(\frac{A_1}{(-\alpha)^{n_1}}+\frac{A_2}{(-\alpha)^{n_2}}\right)\right|\ge \frac{\alpha^{n_2}}{2X}-\frac{2X}{\alpha^{n_2}}.
$$
Assume first that 
\begin{equation}
\label{eq:4}
\frac{\alpha^{n_2}}{2X}-\frac{2X}{\alpha^{n_2}}\le \frac{\alpha^{n_2}}{4X}.
\end{equation}
Then $\alpha^{2n_2}<8X^2$, so $\alpha^{n_2}<3X$. Hence, 
\begin{equation}
\label{eq:upperboundforn2}
n_2<\frac{\log(3X)}{\log \phi},
\end{equation}
which together with \eqref{eq:n1minusn2} gives
\begin{equation}
\label{eq:tttt}
n_4\le n_3\le n_2\le \frac{\log(3X)}{\log \phi}\quad {\text{\rm and}}\quad n_1<\frac{\log(21X^2)}{\log \phi}.
\end{equation}
This was assuming \eqref{eq:4} holds. Otherwise, 
\begin{align*}
\frac{\alpha^{n_2}}{4X}\le \frac{\alpha^{n_2}}{2X}-\frac{2X}{\alpha^{n_2}}\le 4X\alpha^{n_3},
\end{align*}
so 
\begin{align*}
\alpha^{n_2-n_3}\le 16 X^2.
\end{align*}
Hence, we get
\begin{equation}
\label{eq:n2minusn3}
n_2-n_3\le 2\frac{\log(4X)}{\log \phi}.
\end{equation}
We rewrite equation \eqref{eq:1} as 
\begin{equation}
\label{eq:5}
\left|\alpha^{n_3}(A_1\alpha^{n_1-n_3}+A_2\alpha^{n_2-n_3}+A_3)-\left(\frac{A_1}{(-\alpha)^{n_1}}+\frac{A_2}{(-\alpha)^{n_2}}+\frac{A_3}{(-\alpha)^{n_3}}\right)\right|=|-A_4(\alpha^{n_4}-\beta^{n_4})|.
\end{equation}
Assume first that
\begin{equation}
\label{eq:6}
A_1\alpha^{n_1-n_3}+A_2\alpha^{n_2-n_3}+A_3=0.
\end{equation}
Let $i=n_2-n_3,~j=n_1-n_3$. Then
\begin{align*}
j=(n_1-n_2)+(n_2-n_3)\le \frac{\log(112X^3)}{\log \phi} \quad \text{and}\quad  i\le 2\frac{\log(4X)}{\log \phi}
\end{align*} 
are bounded. Thus, one computes all polynomials $A_1X^{j}+A_2X^{i}+A_3$ 
and checks which of them has a root $\alpha$ which is a quadratic unit of norm $-1$. For these Lucas sequences, it is the case that also $\beta$ is a root of the same polynomial so that the left--hand side of 
\eqref{eq:5} is zero for any $n_3$. This shows that also $n_4=0$. Thus, we have that $$(n_1,n_2,n_3,n_4)=(n_3+i,n_3+j,n_3,0)$$
is a parametric family of solutions.  From now on we assume that the expression shown at \eqref{eq:6} is nonzero. Then 
\begin{align*}
|A_1\alpha^{n_1-n_3}+A_2\alpha^{n_2-n_3}+A_3||A_1\beta^{n_1-n_2}+A_2\beta^{n_2-n_3}+A_3|\ge 1.
\end{align*}
The second factor in the left--hand side is $\le 3X$, therefore we conclude that  
\begin{align*}
A_1\alpha^{n_1-n_3}+A_2\alpha^{n_2-n_3}+A_3|\ge \frac{1}{3X}.
\end{align*}
Further,
\begin{align*}
\left|\frac{A_1}{(-\alpha)^{n_1}}+\frac{A_2}{(-\alpha)^{n_2}}+\frac{A_3}{(-\alpha)^{n_3}}\right|\le \frac{3X}{\alpha^{n_3}}.
\end{align*}
Hence, assuming \eqref{eq:6} does not hold, the left--hand side of \eqref{eq:5} is at least as large as 
\begin{align*}
\frac{\alpha^{n_3}}{3X}-\frac{3X}{\alpha^{n_3}}.
\end{align*}
We distinguish two cases. If 
\begin{equation}
\label{eq:7}
\frac{\alpha^{n_3}}{3X}-\frac{3X}{\alpha^{n_3}}\le \frac{\alpha^{n_3}}{6X},
\end{equation}
we then get $\alpha^{2n_3}<18X^2$, so $\alpha^{n_3}\le 5X$. Hence,
\begin{equation}
\label{eq:8}
n_3\le \frac{\log(5X)}{\log \phi}.
\end{equation}
Together with \eqref{eq:n1minusn2} and \eqref{eq:n2minusn3}, we get
\begin{equation}\label{eq:9}
\begin{aligned}
n_4 & \le  n_3\le \frac{\log(5X)}{\log \phi},\\
n_2 & \le  (n_2-n_3)+n_3\le \frac{\log(80X^3)}{\log \phi},\\
n_1 & \le  (n_1-n_2)+n_2\le \frac{\log(560X^4)}{\log \phi}.
\end{aligned}
\end{equation}

Note that \eqref{eq:9} contains \eqref{eq:tttt}. 
Finally assume that \eqref{eq:7} does not hold. Then the left--hand side of \eqref{eq:5} is at least
\begin{align*}
\frac{\alpha^{n_3}}{6X}.
\end{align*}
Comparing with the right--hand side of \eqref{eq:5} we get
\begin{align*}
\frac{\alpha^{n_3}}{6X}\le 2X\alpha^{n_4}\le 2X\alpha^{n_4},
\end{align*}
so $\alpha^{n_3-n_4}\le 12X^2$. Thus,
\begin{equation}
\label{eq:n3minusn4}
n_3-n_4\le \frac{\log(12X^2)}{\log \phi}.
\end{equation}
Finally, we rewrite our equation as 
\begin{equation}
\label{eq:10}
\alpha^{n_4}(A_1\alpha^{n_1-n_4}+A_2\alpha^{n_2-n_4}+A_3\alpha^{n_3-n_4}+A_4)=\beta^{n_4}(A_1\beta^{n_1-n_4}+A_2\beta^{n_2-n_4}+A_3\beta^{n_3-n_4}+A_4).
\end{equation}
The exponents $i=n_3-n_4,~j=n_2-n_4,~k=n_1-n_4$ have only finitely many values. In fact,
\begin{equation*}
\begin{aligned}
i & \le \frac{ \log(12X^2)}{\log\phi},\\
j & =  i+(n_2-n_3)\le \frac{(\log(12X^2)+\log(16X^2))}{\log \phi}\le \frac{\log(200X^3)}{\log \phi},\\
k & =  j+(n_1-n_2)\le \frac{(\log(200X^3)+\log(7X))}{\log \phi}=\frac{\log(1400 X^4)}{\log \phi}.
\end{aligned}
\end{equation*}
So, we take all such polynomials $AX^k+A_2X^j+A_3X^i+A_4$ and search which ones of them have a root $\alpha$ which is a quadratic unit of norm $-1$. For such,  \eqref{eq:10} holds 
for all $n_4$. Hence, we got the parametric family 
$$(n_1,n_2,n_3,n_4)=(n_4+k,n_4+j,n_4+i,n_4).
$$ 
Assume next the the left--hand side of \eqref{eq:10} is nonzero. 
Then 
\begin{align*}
|A_1\alpha^{n_1-n_4}+A_2\alpha^{n_2-n_4}+A_3\alpha^{n_3-n_4}+A_4||A_1\beta^{n_1-n_4}+A_2\beta^{n_2-n_4}+A_3\beta^{n_3-n_4}+A_4|\ge 1.
\end{align*}
The second factor on the left--hand side above is $\le 4X$. Hence, 
\begin{align*}
|A_1\alpha^{n_1-n_4}+A_2\alpha^{n_2-n_4}+A_3\alpha^{n_3-n_4}+A_4|\ge \frac{1}{4X}.
\end{align*}
Hence, in \eqref{eq:10}, we get
\begin{align*}
\frac{\alpha^{n_4}}{4X}\le 4X|\beta|^{n_4}=\frac{4X}{\alpha^{n_4}},
\end{align*}
which gives 
\begin{equation}
\label{eq:n4}
n_4\le \frac{\log(4X)}{\log \phi}.
\end{equation}
This together with \eqref{eq:n1minusn2}, \eqref{eq:n2minusn3} and \eqref{eq:n3minusn4} gives
\begin{equation}\label{eq:final}
\begin{aligned}
n_4 & \le  \frac{\log(4X)}{\log \phi},\\
n_3 & \le  (n_3-n_4)+n_4\le \frac{\log(50X^3)}{\log \phi},\\
n_2 &\le  (n_2-n_3)+n_3\le \frac{\log(1000X^5)}{\log \phi},\\
n_1 & \le (n_1-n_2)+n_2<\frac{\log(10000X^6)}{\log \phi}.
\end{aligned}
\end{equation}
Note that \eqref{eq:final} contains \eqref{eq:9} and \eqref{eq:CDarezero}.
Recalling that we have to replace $X$ by $2X$, we got the following theorem which is our second result.

\begin{theorem}\label{main1}
Let $\phi:=(1+{\sqrt{5}})/2$ be the smallest possible  $ \alpha $. Relabeling the variables $(n,m,n_1,m_1)$ to $(n_1,n_2,n_3,n_4)$, where $n_1\ge n_2\ge n_3\ge n_4$. If $n_1=n_2$, we rewrite the Diophantine equation \eqref{eq:main} as 
\begin{align*}
(A-C)U_n+BU_m=DU_{m_1},
\end{align*}
and change $(A,B,C,D)$ to $(A-C,B,D, 0)$. 
Thus, $n_1>n_2$. Furthermore, we change the sign of some of the coefficients $(A,B,C,D)$ so that the Diophantine equation \eqref{eq:main} becomes
\begin{equation}
\label{eq:main1}
A_1U_{n_1}+A_2U_{n_2}+A_3U_{n_3}+A_4U_{n_4}=0.
\end{equation}
Assume $r\le 14X$. Then, the solutions of the Diophantine equation \eqref{eq:main1} are of two types:
\begin{itemize}
\item[(i)] Sporadic ones. These are finitely many and they satisfy:
\begin{equation*}
\begin{aligned}
n_4  &\le  \frac{\log(8X)}{\log \phi},\quad n_3\le \frac{\log(400X^3)}{\log \phi},\\
n_2 & \le  \frac{\log(32000X^5)}{\log \phi},\quad n_1\le \frac{\log(640000X^6)}{\log \phi}.
\end{aligned}
\end{equation*}
\item[(ii)] Parametric ones. These are of one of the two forms:
\begin{align*}
(n_1,n_2,n_3,n_4)=(n_3+j,n_3+i,n_3,0),
\end{align*}
where 
\begin{align*}
i\le 2\frac{\log(8X)}{\log \phi}\quad {\text{and}}\quad j\le \frac{\log(500X^3)}{\log \phi},
\end{align*}
and $\alpha$ is a root of $A_1X^i+A_2X^j+A_3=0$, or of the form
\begin{align*}
(n_1,n_2,n_3,n_4)=(n_4+k,n_4+j,n_4+i,n_4),
\end{align*}
where 
\begin{align*}
i\le \frac{\log(50X^2)}{\log \phi},\quad j\le \frac{\log(1600X^3)}{\log \phi},\quad k\le \frac{\log(25000X^4)}{\log \phi},
\end{align*} 
and $\alpha$ is a root of 
\begin{align*}
A_1X^k+A_2X^j+A_3X^i+A_4=0.
\end{align*}
\end{itemize}
\end{theorem}

\section{Numerical examples}

Just for fun, we took $A_1,A_2,A_3,A_4\in \{0,\pm 1\}$. Hence, $X=1$, therefore $r\le 14$. Thus, Theorem \ref{main1} says that the  sporadic solutions are of the form 
$$
U_{n_1}\pm A_2U_{n_2}\pm A_3U_{n_3}\pm A_4U_{n_4}=0,\quad A_2,A_3,A_4\in \{0,\pm 1\},~n_1>n_2\ge n_3\ge n_4\ge 0.
$$
Here, $n_4\le 4,~n_3\le 12,~n_2\le 21$ and $n_1>n_2$. To search for them, we searched for $r\in [1,14]$, $n_4\in [0,4],~n_3\in [n_4,12],~n_2\in [n_3,21]$, $\varepsilon_4\in \{0,1\}$, $\varepsilon_3\in \{0,\pm 1\}$, $\varepsilon_2\in \{0,\pm 1\}$ such that 
$$
U_{n_1}=|\varepsilon_2 U_{n_2}+\varepsilon_3U_{n_3}+\varepsilon_4U_{n_4}|\quad {\text{\rm holds~for~some}}\quad n_1>n_2.
$$
A Mathematica code running for a few seconds found $207$ solutions. Of them $194$ correspond to the Fibonacci sequence ($r=1$), $12$ correspond to the Pell sequence ($r=2$) and only one of them 
namely $U_1+U_1+U_1=U_2$ corresponds to $r=3$. For parametric ones, Theorem \ref{main1} says that we need to find positive integers $i\le 8,~j\le 15,~k\le 21$ such that $X^k+\varepsilon_1X^j+\varepsilon_2X^i+\varepsilon_3$ is a multiple of 
$X^2-rX-1$ for some $r\in [1,14]$, where $\varepsilon_1\in\{0,\pm 1\},~\varepsilon_2\in \{0,\pm 1\},~\varepsilon_3\in \{\pm 1\}$. The only such instances found were $r=1$ for which only $X^2-X-1$ and $X^4-X^3-X-1$ were multiples of $X^2-rX-1=X^2-X-1$. These two instances lead to the parametric families 
$$
F_{n+2}-F_{n+1}-F_n-F_0=0\quad {\text{\rm and}}\quad F_{n+4}-F_{n+3}-F_{n+1}-F_n=0,
$$
which hold for all $n\ge 0$. Enlarging $X$ (so, say allowing $A_1,A_2,A_3,A_4$ in $[-X,X]$, $A_1\ne 0$ for a fixed integer $X\ge 2$) would of course detect more sporadic solutions and more parametric families involving the Pell sequence, etc. We leave pursuing such numerical investigations for the interested reader.

\section{Comments}
In this paper, we worked with the Lucas sequence $(U_n)_{n\ge 0}$ of characteristic equation $X^2-rX-1=0$, where $r\ge 1$ is also a variable. Similar arguments can be used to deal with the equation \eqref{eq:main} when 
the characteristic equation of $(U_n)_{n\ge 0}$ is $X^2-rX-s=0$, where $s$ is a fixed nonzero integer. The conclusion should be the same, namely that for given $A,B,C,D$, equation \eqref{eq:main} implies that all its solutions
come in two flavours; namely sporadic (maybe none) solutions whose indices $\max\{n,n_1\}$ are bounded by a computable function $f(X,s)$, depending on $X$ and $s$; and possibly additional parametric solutions namely of the form 
$(n,m,n_1,m_1)=(n,n-i,n-j,n-k)$, where $i,j,k$ are bounded by some computable function $g(X,s)$ depending on $X$ and $s$, and $n$ is a free parameter. Again, we leave pursuing such endeavours to the interested reader.

\section*{acknowledgements}
Open Access funding provided by Austrian Science Fund (FWF). M.~D. and R.~T. were supported by the FWF projects: F5510-N26 -- Part of the special research program (SFB), ``Quasi-Monte Carlo Methods: Theory and Applications'' and W1230 --``Doctoral Program Discrete Mathematics''. F.~L. was supported by grant RTNUM19 from CoEMaSS, Wits, South Africa.

 \end{document}